\documentclass[12pt,reqno]{amsart}
\usepackage[utf8]{inputenc}
\usepackage{setspace}
\usepackage{enumerate}
\usepackage{graphicx}
\usepackage{yfonts}
\usepackage{amssymb}
\usepackage{amsmath}
\usepackage{amsthm}
\usepackage{tikz}
\usepackage{tikz-cd}
\usetikzlibrary{arrows}
\usepackage{enumitem}
\usepackage{url}

\calclayout

\theoremstyle{plain}

\newtheorem*{theorem*}{Theorem}
\newtheorem*{proposition*}{Proposition}
\newtheorem*{remark*}{Remark}
\newtheorem*{lemma*}{Lemma}
\newtheorem{theorem}{Theorem}[section]
\newtheorem{definition}[theorem]{Definition}

\newtheorem{lemma}[theorem]{Lemma}
\newtheorem{example}[theorem]{Example}
\newtheorem{conjecture}[theorem]{Conjecture}
\newtheorem{remark}[theorem]{Remark}

\title{Roots of Unity and Projective Equivalence}
\author{David Hubbard}
\date{February 2024}

\begin{document}

\begin{abstract}
    We say that sets $S$ and $S'$ of points in $\mathbb{P}^1$ are \textit{projectively equivalent} if there exists a projective automorphism $\gamma$ such that $\gamma(S)=S'$. In \cite{fu2022projective}, Fu proves that, if $S$ and $S'$ both consist of roots of unity, then the largest possible cardinality for $S$ is 14 (after discounting automorphisms $\gamma$ that act bijectively on the set of all roots of unity $\gamma$). Moreover, Fu constructs the two possible maximal sets, which are unique up to projective equivalence. In this article, we given an elementary proof that the cardinality of two such sets is at most 18 using the methods of Beukers and Smyth \cite{beukers2002cyclotomic}. Moreover, we show precisely how their method fails to give the tightest bound in the maximal cases of Fu.
\end{abstract}

\maketitle
\onehalfspacing

\section{Introduction}\label{sIntroduction}

In \cite{fu2022projective}, Fu proves the following:
\begin{theorem}\label{tFuTheorem}
    Let $\mu_\infty\subseteq \mathbb{P}^1$ denote the roots of unity. Suppose $S,S'\subseteq \mu_{\infty}$ are projectively equivalent i.e.\@ there exists a projective automorphism $\gamma$ such that $\gamma(S)=S'$. Suppose further that $\gamma(\mu_{\infty})\neq \mu_{\infty}$. Then:
    \[|S|\leq 14\]
    Moreover, up to projective equivalence, there are two possibilities for $|S|=14$:
    \begin{enumerate}
        \item 
            \[S_1 =\{1,\zeta,\zeta^2,\zeta^3,\zeta^4,\zeta^5,\zeta^6,\zeta^9,\zeta^{11},\zeta^{14},\zeta^{18},\zeta^{22},\zeta^{25},\zeta^{27}\}\]
        where $\zeta$ is a primitive $30$-th root of unity;

        \item 
        \[S_2 = \{1,\omega,\omega^{2},\omega^{4},\omega^6,\omega^8,\omega^{12},
            \omega^{22},\omega^{31},\omega^{40}, \omega^{50},\omega^{54},\omega^{56}, \omega^{58}\}\]
        where $\omega$ is a primitive $60$-th root of unity;
        
    \end{enumerate}
\end{theorem}

Fu's method is algorithmic in nature, however does provide explicitly the sharpest bounds, and computes the maximal examples. In this article, we will provide a completely abstract proof that the cardinality of (non-trivially) projectively equivalent sets of roots of unity are uniformly bounded, which is comparatively much shorter than Fu's exposition, with the caveat that we obtain a less sharp bound of 18. 

\begin{remark}\label{rTorsionPointsofGm}
    Let $\mathbb{G}_{m}$ denote the algebraic torus over $\mathbb{C}$. Then the torsion subgroup $\mathbb{G}_{m}[\infty]$ is given by $\mu_{\infty}$. To view the torsion subgroup in $\mathbb{P}^1$, we make a choice of embedding $\mathbb{G}_{m}\hookrightarrow \mathbb{P}^{1}$ i.e.\@ a choice of homogeneous coordinates:
    \begin{align*}
        \mathbb{G}_{m} &\longrightarrow \mathbb{P}^{1}\\
        z &\longmapsto [1:z]
    \end{align*}
    Any such embedding is unique up to a projective automorphism of $\mathbb{P}^{1}$ (i.e.\@ a change of coordinates).
\end{remark}

In view of this remark, we reframe Fu's result as follows: Let $\iota_{j}:\mathbb{G}_{m}\hookrightarrow \mathbb{P}^{1}$ ($j=1,2$) be two embeddings as above. Then there exists a uniform bound on the number of common torsion points:
\[\iota_1(\mathbb{G}_{m}[\infty])\cap \iota_2(\mathbb{G}_{m}[\infty])\]
as long as the intersection is not infinite.

The analogue of this result for elliptic curves was initially discussed by Bogomolov, Fu and Tschinkel \cite{bogomolov2018torsion}. Suppose that $E$ is an elliptic curve over $\mathbb{C}$ and denote by $\iota$ the involution of $E$. Then there exists a double cover $\pi:E\xrightarrow{}\mathbb{P}^1$ such that $\pi \circ \iota = \pi$. Again, the choice of such a double cover is unique up to projective automorphism. Given two such double covers $\pi_j:E_j\xrightarrow{}\mathbb{P}^{1}$ ($j=1,2$), Bogomolov, Fu and Tschinkel conjectured the existence of a uniform bound on the number of `common' torsion points on $E_1$ and $E_2$ in $\mathbb{P}^1$:
\begin{conjecture}
    There exists a constant $c>0$, independent of $E_j$ and $\pi_j$, such that:
    \[\left| \pi_1(E_1[\infty])\cap \pi_2(E_2[\infty])\right| \leq c\]
    for all pairs $(E_j,\pi_j)$ where this intersection is finite.
\end{conjecture}

It is easy to give criteria for when this intersection is finite. Consider the curve $X:=(\pi_1 \times \pi_2)^{-1}(\Delta) \subseteq E_1 \times E_2$ where $\Delta \subseteq \mathbb{P}^1\times \mathbb{P}^1$ is the diagonal. In \cite{bogomolov2018torsion}, it is shown that $X$ has geometric genus $g=5-\delta$ where $\delta:= |\pi_1(E_1[2])\cap \pi_2(E_2[2])|\in \{0,1,2,3,4\}$. In particular, when $\pi_1(E_1[2])\neq \pi_2(E_2[2])$, $X$ is integral of genus $g\geq 2$ and one has that
\[X\cap (E_1[\infty]\times E_2[\infty])\]
is finite by the Manin-Mumford conjecture. Clearly, it is equivalent to uniformly bound the torsion points on the curves $X$ induced by pairs $(E_j, \pi_j)$. A uniform bound was given by DeMarco, Krieger and Ye \cite{demarco2020uniform} for the $g=2$ case and the general case has a uniform bound by Poineau \cite{poineau2022dynamique}, for example. However, the bounds obtained are not minimal or effective unlike the bounds of Fu in the case of $\mathbb{G}_{m}$, and finding an effective uniform bound is still an open problem. 

Returning to the case of $\mathbb{G}_m$, we analogously reframe our problem by counting torsion points on a curve $X\hookrightarrow (\mathbb{G}_{m})^{2}$ where $X=(\iota_1 \times \iota_2)^{-1}(\Delta)$ using the notation of Remark \ref{rTorsionPointsofGm}, and $\Delta \subseteq \mathbb{P}^1\times \mathbb{P}^1$ is the diagonal. Counting torsion points/pairs of roots of unity on an algebraic curve in the $2$-torus has been studied extensively by Beukers and Smyth in \cite{beukers2002cyclotomic}. From this perspective, it is easy to deduce a uniform bound.

\begin{remark}\label{rMethodsofBeukersSmyth}
    In this article we illustrate in full the method of Beukers and Smyth for our curve $X$. We do this for two reasons: We are able to make a slight refinement on the bounds one obtains from applying their method directly (one would obtain a bound of $22$ which we improve to $18$). Moreover, we are able to see why their method fails to provide the sharpest bound in our specific case which we discuss in Section \ref{sMaximalCases}.
\end{remark}

\section*{Acknowledgments}

The author was funded through the Warwick Mathematics Institute, Centre for Doctoral Training, and the UK Engineering and Physical Sciences Research Council studentship (EP/W523793/1).

\section{Roots of Unity and Projective Equivalence}\label{sRootsofUnity}

Recall that the automorphisms of $\mathbb{P}^{1}$ are given by the M\"obius transformations:
\[\textnormal{M\"ob}(\mathbb{C}) = \biggl\{ \frac{ax+b}{cx+d} \mid ad-bc \neq 0 \biggr\}\]

Let $H$ denote the subgroup of $\textnormal{M\"ob}(\mathbb{C})$ generated by:
\[H=\Bigl\langle \zeta x, \frac{1}{x} \mid \zeta \in \mu_{\infty}\Bigr\rangle \]

We will prove the following theorem:
\begin{theorem}\label{tMainTheorem}
    Let $\gamma \in \textnormal{M\"ob}(\mathbb{C}) \setminus H$. Then we have the following uniform bound:
    \[| \gamma(\mu_{\infty})\cap \mu_{\infty} | \leq 18\]

    Moreover, if $\gamma \in H$ then $\gamma(\mu_{\infty})\cap \mu_{\infty}$ is infinite.
\end{theorem}

We will use the following set-up:

Let $\Delta \subseteq \mathbb{P}^{1}\times \mathbb{P}^{1}$ be the diagonal. Using homogeneous coordinates $([x_{0}:x_{1}],[y_{0}:y_{1}])$, the diagonal is given by $\Delta=\{G=0\}$ where:
\[G=x_{0}y_{1}-x_{1}y_{0}\]
Write $\gamma(x_{0},x_{1})=[ax_{0}+bx_{1}:cx_{0}+dx_{1}]$ for a M\"obius map $\gamma$ and let:
\[F=(ax_{0}+bx_{1})y_{1}-(cx_{0}+dx_{1})y_{0}\]
Then the curve $X=\{F=0\}$ fits into the following Cartesian diagram:
\[\begin{tikzcd}
	X & {\mathbb{P}^{1}\times \mathbb{P}^{1}} \\
	\Delta & {\mathbb{P}^{1}\times \mathbb{P}^{1}}
	\arrow[hook, from=1-1, to=1-2]
	\arrow[from=1-1, to=2-1]
	\arrow[hook, from=2-1, to=2-2]
	\arrow["{\gamma \times \text{Id}}", from=1-2, to=2-2]
\end{tikzcd}\]
where the vertical arrows are isomorphisms. Theorem \ref{tMainTheorem} is now equivalent to the following:
\begin{theorem}\label{tEquivalentTheorem}
    Let $\gamma \in \textnormal{M\"ob}(\mathbb{C}) \setminus H$ and $X=(\gamma \times \textnormal{Id})^{-1}(\Delta)$. Then we have the following uniform bound:
    \[| X \cap \mu_{\infty}^{2} | \leq 18\]
\end{theorem}

We first note that it is necessary to exclude elements of the subgroup $H$:

\begin{definition}\label{dTranslatesofCurves}
    Let $Y=\{g=0\} \subseteq (\mathbb{C}^{*})^2$ where $g=g(x,y)$ is a polynomial. Suppose $(a,b)\in (\mathbb{C}^{*})^2$. We define the \textbf{translate} of $Y$ by $(a,b)$ to be the curve:
    \[(a,b)\cdot Y =\{g(a^{-1}x,b^{-1}y)=0\}\]
\end{definition}

\begin{lemma}\label{lToricAutomorphisms}
    Let $\gamma \in H$ and $X=(\gamma \times \textnormal{Id})^{-1}(\Delta)$. Then $X\cap (\mathbb{C}^{*})^{2}$ is the translate of a sub-torus by a pair of roots of unity; and thus $X\cap\mu_{\infty}^{2}$ is infinite \textnormal{(\cite{lang2013fundamentals} (Ch.~9))}.
\end{lemma}
\begin{proof}
    Using affine coordinates $x=x_{0}/x_{1}$ and $y=y_{0}/y_{1}$ we have that $X\cap(\mathbb{C}^{*})^{2}$ is defined by:
    \[f=(ax+b)-(cx+d)y\]
    All embedded sub-tori are given (up to scaling) by polynomials of the form:
    \[x^{m}-y^{n}=0 \quad \textnormal{or} \quad x^{m}y^{n}-1=0\]
    where $m$ and $n$ are coprime positive integers.
    Thus, translates by roots of unity are all (up to scaling) of the form:
    \[x^{m}-\zeta y^{n}=0 \quad \textnormal{or} \quad x^{m}y^{n}-\zeta=0\]
    Applying this to $f$, we see that it gives a translate of a sub-torus if and only if $\gamma \in H$.
\end{proof}

To obtain an overall bound we will need to split into cases depending on the extension of $\mathbb{Q}$ over which:
 \[f=(ax+b)-(cx+d)y\]
 is defined. For such a field to be well-defined, we constrain that (after scaling) at least one coefficient of $f$ is 1. Then we set $K=\mathbb{Q}(a,b,c,d)$. Write $\mathbb{Q}^{ab}=\mathbb{Q}(\mu_{\infty})$ for the maximal abelian extension of $\mathbb{Q}$.

\underline{Case (i)}:
Suppose first that $K \not\subseteq \mathbb{Q}^{ab}$. Let $L=\mathbb{Q}^{ab}(a,b,c,d)$. Then there exists a conjugate $\sigma: L \hookrightarrow \mathbb{C}$ such that $\sigma \neq \textnormal{Id}$ but $\sigma|_{\mathbb{Q}^{ab}}=\textnormal{Id}$. Then we have that $f^{\sigma}$ and $f$ are distinct irreducible polynomials and that:
\[X\cap \mu_{\infty}^{2} \subseteq X\cap X^{\sigma} \]
where $X\cap X^{\sigma}$ is a finite set of at most 4 points by B\'ezout's theorem.

\section{Uniform Bounds and Convex Geometry}\label{sConvexGeometry}

We may now assume that $K\subseteq \mathbb{Q}^{ab}$. As $K$ is a finite abelian extension of $\mathbb{Q}$, there exists an $N$-th root of unity $\zeta_{N}$ such that $K= \mathbb{Q}(\zeta_{N})$. In this section, we describe the methods of \cite{beukers2002cyclotomic} on bounding the number of pairs of roots of unity on our given curve $X$. We make use of the following definitions from convex geometry which will allow us to rule out certain degenerate cases:

\begin{definition}\label{dConvexDefinitions}
    Let $g=\sum_{(i,j)\in\mathbb{Z}^{2}}a_{ij}x^{i}y^{j}$ be a Laurent polynomial.
    \begin{enumerate}
        \item The \textbf{support} of $g$ is the set of lattice points:
        \[\textnormal{Supp}(g)=\{(i,j)\in \mathbb{Z}^{2}\mid a_{ij}\neq 0\}\]
        \item The \textbf{lattice} associated to $g$ is the abelian group generated by the set of differences:
        \[\mathcal{L}(g)=\bigl\langle (i,j)-(i',j') \mid (i,j),(i',j')\in \textnormal{Supp}(g) \bigr\rangle\]
        \item The rank of $\mathcal{L}(g)$ is either 1 or 2. When the rank is 2 we say $\mathcal{L}(g)$ is \textbf{full} if it generates $\mathbb{Z}^{2}$.
        \item The \textbf{Newton polytope} of $g$ is the convex hull of its support:
        \[\textnormal{Newt}(g)=\textnormal{Conv}(\textnormal{Supp}(g))\]
        \item The \textbf{volume} of $f$ is the volume of its Newton polytope:
        \[\textnormal{Vol}(g)=\textnormal{Vol}(\mathcal{N}(g))\]
    \end{enumerate}
\end{definition}
\underline{Case (ii)}: Suppose that $\mathcal{L}(f)$ has rank 1. Then it follows that at least two of $a,b,c,d$ are 0. Using the condition $ad-bc=1$, it follows that $f$ is of the form (up to scaling):
\[f(x,y)=xy+b \quad \textnormal{or} \quad f(x,y)=x-dy\]
Then, either both $b$ and $d$ are roots of unity and $X$ is the translate of a torus and contains infinitely many roots of unity. Otherwise, one of $b$ or $d$ is not a root of unity and $X$ contains no roots of unity.

\medbreak

Henceforth, we may assume that $\mathcal{L}(f)$ is rank 2. It is immediate that the lattice $\mathcal{L}(f)$ is full, as it contains at least two of $(1,0),(0,1),(1,1)$. 

\begin{remark}\label{rSubToriLattices}
    If $X$ were the translate of a sub-torus by a root of unity, then $\mathcal{L}(f)$ would be of rank 1. So henceforth, we may assume that $X\cap \mu_{\infty}^2$ is finite.
\end{remark}

\underline{Case (iii)}: We now underline the general method for bounding roots of unity on $X$. To make the exposition more clear, we will assume for now that $f\in \mathbb{Q}[x,y]$. Then we will show how the argument can be furnished for the general case in (iv) and (v).

Consider the following collection of polynomials:
\begin{align*}
    &f_{1}(x,y)=f(x,-y)=(ax+b)+(cx+d)y \\
    &f_{2}(x,y)=f(-x,y)=(-ax+b)-(-cx+d)y \\
    &f_{3}(x,y)=f(-x,-y)=(-ax+b)+(-cx+d)y \\
    &f_{4}(x,y)=f(x^{2},y^{2})=(ax^{2}+b)-(cx^{2}+d)y^{2} \\
    &f_{5}(x,y)=f(x^{2},-y^{2})=(ax^{2}+b)+(cx^{2}+d)y^{2} \\
    &f_{6}(x,y)=f(-x^{2},y^{2})=(-ax^{2}+b)-(-cx^{2}+d)y^{2} \\
    &f_{7}(x,y)=f(-x^{2},-y^{2})=(-ax^{2}+b)+(-cx^{2}+d)y^{2}
\end{align*}

Let $X_{i}=\{f_{i}=0\}$ for $i=1,...,7$. We claim that, with our assumptions on $f$, $X$ intersects each $X_{i}$ in a finite set (i.e.~ $f$ and $f_{i}$ are coprime), and if $(\zeta,\zeta')\in X\cap \mu_{\infty}^2$, then $(\zeta,\zeta')\in X_{i}$ for some $i$. Then it follows:
\[|X\cap \mu_{\infty}^{2}| \leq \sum_{i=1}^{7}|X\cap X_{i}|\]

To prove the second point we use the following lemma:
\begin{lemma}\label{lConjugatesofRootsofUnity}
    Suppose $\zeta_{N}$ is an $N$-primitive root of unity. Then it is conjugate to:
    \begin{align*}
        -\zeta_{N} \quad &\textnormal{if} \quad N\equiv 0 ~~ (\textnormal{mod} ~2)\\
        \zeta_{N}^{2} \quad &\textnormal{if} \quad N\equiv 3 ~~ (\textnormal{mod} ~4)\\
        -\zeta_{N}^{2} \quad &\textnormal{if} \quad N\equiv 1 ~~ (\textnormal{mod} ~4)
    \end{align*}
\end{lemma}

If $(\zeta,\zeta')$ are roots of unity on $X$, then for some $N>0$ we can write:
\[(\zeta,\zeta')=(\zeta_{N}^{a},\zeta_{N}^{b})\]
Then, by the above, there exists a conjugate $\sigma:\mathbb{Q}(\zeta_{N})\xrightarrow{}\mathbb{Q}(\zeta_{N})$ sending:
$$\zeta_{N}\mapsto -\zeta_{N},\zeta_{N}^{2} ~~ \textnormal{or} ~~ -\zeta_{N}^{2}$$
Therefore $\sigma(f(\zeta_{N}^{a},\zeta_{N}^{b}))=f_{i}(\zeta_{N}^{a},\zeta_{N}^{b})=0$ for at least one $i$.

To show that $X$ has no common component with $X_{i}$ we note, for $i=1,2,3$, that $f\mid f_{i}$ would imply $c=d=0$, $a=c=0$ or $a=d=0$ respectively, so again contradicting $ad-bc \neq 0 $ or that $\mathcal{L}(f)$ is of rank 2.

So assuming that $f,f_{1},f_{2},f_{3}$ have no common factors, if $f|f_{4}$ then similarly $ff_{1}f_{2}f_{3}|f_{4}$ which is a contradiction, by considering degrees. A similar argument shows that $f\nmid f_{i}$ for $i=5,6,7$.

Finally, to get a sharper bound on the intersection numbers, we use a \textit{toric} variant of B\'ezout's theorem, using mixed volumes (\cite{fulton1993introduction} p.~122):

\begin{lemma}\label{lToricBezout}
    Suppose $f(x,y)$ and $g(x,y)$ are Laurent polynomials over $\mathbb{C}$ such that $X=\{f=0\}$ and $Y=\{g=0\}$ have no common component. Then:
    \[|X\cap Y| \leq \textnormal{Vol}(\textnormal{Newt}(f)+\textnormal{Newt}(g))-\textnormal{Vol}(f)-\textnormal{Vol}(g)\]
\end{lemma}

Applying this to $X$ and $X_{i}$ we have for $i=4,5,6,7$ that:
\begin{align*}
    |X\cap X_{i}| \leq &\textnormal{Vol}(3\textnormal{Newt}(f))-\textnormal{Vol}(f)-4\textnormal{Vol}(f) \\
    \leq &9\textnormal{Vol}(f)-5\textnormal{Vol}(f)\\
    =& 4\textnormal{Vol}(f) \leq 4
\end{align*}

Note that $|X\cap X_{i}|\leq 2$ for $i=1,2,3$, by B\'ezout's theorem. Thus, we obtain the bound of:
\[|X\cap \mu_{\infty}^{2}| \leq \sum_{i=1}^{7}|X\cap X_{i}|\leq 2+4\times 4 =22\]

\begin{remark}\label{rRefinements}
    The bound above is the one obtained directly from the method of Beukers and Smyth. However, note that in our case, for $i=1,2$ one can check that $X\cap X_{i}$ does not lie in $(\mathbb{C}^{*})^{2}$. For example:
    \begin{align*}
        X\cap X_{1}=&\{ax+b=0, (cx+d)y=0\}\\
        =& \{(-b/a,0)\}
    \end{align*}
    Thus overall, we have the bound:
    \[|X\cap \mu_{\infty}^{2}| \leq \sum_{i=3}^{7}|X\cap X_{i}|\leq 2+4\times 4 =18\]
    (cf.\@ Remark \ref{rMethodsofBeukersSmyth}).
\end{remark}

To finish the proof of Theorem \ref{tEquivalentTheorem}, we need to show that the same bound holds for the more general case $f\in \mathbb{Q}^{ab}[x,y]$. To do this we make the following:

Recall we take $\mathbb{Q}\subseteq K \subseteq \mathbb{Q}^{ab}$ to be the field of definition of $f$ and so $K=\mathbb{Q}(\zeta_{N})$ for some $N>0$. It is clear that $X$ contains the same number of pairs of roots of unity as any translate (cf.\@ Def.\@ \ref{dTranslatesofCurves}) $\{f(\zeta x,\zeta' y)=0\}\cap \mu_{\infty}^{2}$ for a pair of roots of unity $(\zeta,\zeta')$. So let $\zeta_{N}$ be an $N$-th root of unity where $N$ is \textit{minimal} over all translates $f(\zeta x,\zeta' y)$, with field of definition $\mathbb{Q}(\zeta_{N})$. Henceforth, we replace $f$ by this translate.

We then have two cases: either $N$ is odd or $4\mid N$. This is because, if $N=2M$ where $M$ is odd, then $\zeta_{N}=-\zeta_{M}$ for some primitive $M$-th root and $N$ is not minimal.

\underline{Case (iv)}: $N$ odd

Then, by Lemma \ref{lConjugatesofRootsofUnity}, $\sigma:\zeta_{N} \mapsto \zeta_{N}^{2}$ is an automorphism. We again consider a collection of curves $X_{i}=\{f_{i}=0\}$, but replace $f_{4},f_{5},f_{6},f_{7}$ with $f_{4}^{\sigma},f_{5}^{\sigma},f_{6}^{\sigma},f_{7}^{\sigma}$. One similarly argues with lattices that we still have $f_{i} \nmid f$ for all $i$. However, it is non-trivial to check that each pair of roots of unity $(\zeta,\zeta')\in X$ lies on one $X_{i}$.

To see this, suppose $(\zeta,\zeta')\in X\cap \mu_{\infty}^{2}$. Without loss of generality, one can write:
\[(\zeta,\zeta')=(\zeta_{n}^{r},\zeta_{n}^{s})\]
where $(r,s)=1$. If $4 \nmid n$ then $\sigma$ extends to an automorphism:
\begin{align*}
    \tau: K(\zeta_{n})&\longrightarrow K(\zeta_{n})\\
    \zeta_{N} &\mapsto \zeta_{N}^{2}\\
    \zeta_{n} &\mapsto \pm \zeta_{n}^{2}
\end{align*}
depending on $n$ mod 4. Thus:
\[0=\tau(f(\zeta_{n}^{r},\zeta_{n}^{s}))=f_{i}(\zeta_{n}^{r},\zeta_{n}^{s})\]
for one of $i=4,5,6,7$.

If $4 \mid n$ we set $4k=\textnormal{lcm}(n,N)$ and instead consider the automorphism:
\begin{align*}
    \tau: K(\zeta_{n})&\longrightarrow K(\zeta_{n})\\
    \zeta_{N} &\mapsto \zeta_{N}^{2k+1}=\zeta_{N}\\
    \zeta_{n} &\mapsto \zeta_{n}^{2k+1}=-\zeta_{n}
\end{align*}
Then $(\zeta,\zeta')$ lies on $X_{i}$ for $i=1,2$ or 3.

We then have again that:
\[|X\cap \mu_{\infty}^{2}| \leq \sum_{i=3}^{7}|X\cap X_{i}|\leq 2+4\times 4 =18\]

\underline{Case (v)}: $4\mid N$

Suppose again that $(\zeta,\zeta')=(\zeta_{n}^{r},\zeta_{n}^{s})$ is some pair of roots of unity on $X$. As $4\mid N$ we can take $4k=\textnormal{lcm}(n,N)$. We define an automorphism:

\begin{align*}
    \tau: K(\zeta_{n})&\longrightarrow K(\zeta_{n})\\
    \zeta_{N} &\mapsto \zeta_{N}^{2k+1}=\pm\zeta_{N}\\
    \zeta_{n} &\mapsto \zeta_{n}^{2k+1}=\pm\zeta_{n}
\end{align*}

Consider the set of polynomials:
\[\{f_{1},f_{2},f_{3},f^{\tau},f_{1}^{\tau},f_{2}^{\tau},f_{3}^{\tau}\}\]
The polynomials above are dependent on $\tau$ which is dependent on $(\zeta,\zeta')$. However as a \textit{set}; they are independent of this point (they are permuted).

Again, we have that $f$ does not have a common factor with any of these polynomials with the exception that if  $\tau(\zeta_{N})=\zeta_{N}$ then $f^{\tau}=f$ and we don't consider this intersection. On the other hand, if $\tau(\zeta_{N})=-\zeta_{N}$ then $f^{\tau}\neq f$ as otherwise $N$ would not be minimal. We are left with the following possibilities:

\begin{enumerate}[label=(\alph*)]
    \item We cannot have both $\tau(\zeta_{N})=\zeta_{N}$ and $\tau(\zeta_{n})=\zeta_{n}$ as $4k$ would no longer be the lowest common multiple.
    \item If we have:
\[\tau(\zeta_{n})=\zeta_{n} \quad \tau(\zeta_{N})=-\zeta_{N}\]
    then $(\zeta_{n}^{r},\zeta_{n}^{s})$ lies on $f^{\tau}$.
    \item If we have:
\[\tau(\zeta_{n})=-\zeta_{n} \quad \tau(\zeta_{N})=\zeta_{N}\]
    then $(\zeta_{n}^{r},\zeta_{n}^{s})$ lies on $f_{1}$,$f_{2}$ or $f_{3}$.
    \item Finally if we have:
\[\tau(\zeta_{n})=-\zeta_{n} \quad \tau(\zeta_{N})=-\zeta_{N}\]
    then $(\zeta_{n}^{r},\zeta_{n}^{s})$ lies on one of $f_{1}^{\tau}$,$f_{2}^{\tau}$ or $f_{3}^{\tau}$.
\end{enumerate}

Replacing $f_{4},f_{5},f_{6},f_{7}$ with $f^{\tau},f_{1}^{\tau},f_{2}^{\tau},f_{3}^{\tau}$ we have a bound of:
\[|X\cap \mu_{\infty}^{2}| \leq \sum_{i=3}^{7}|X\cap X_{i}|\leq 5\times 2 =10\]

This completes the proof of Theorem \ref{tEquivalentTheorem}, and shows that the bound of 18 holds in all cases.
\section{Maximal Examples}\label{sMaximalCases}

We now show precisely how the method of Beukers and Smyth fails to be sharp for the sets $S_{1},S_{2}$ from Theorem \ref{tFuTheorem}. Namely, we describe automorphisms $\gamma_j$ ($j=1,2$) such that $\gamma_j(S_j)$ consists of roots of unity. We then construct the curves $X$ and  $X_i$, as detailed in the method, and describe how the intersections $X\cap X_i$ overcount the roots of unity lying on $X$.

\begin{example}
Let $\zeta$ be a primitive 30-th root of unity. We consider the matrix:
\[\gamma_1 = \begin{bmatrix}
a & b\\
c & d
\end{bmatrix} = \begin{bmatrix}
-\zeta^7-\zeta^6+\zeta^2 & \zeta^7-\zeta^2\\
1 & -\zeta^6-1
\end{bmatrix}\]

Note that, in $\mathbb{Q}(\zeta)$ we have the identity:
\[\zeta^9 = -\zeta^7-\zeta^6+\zeta^3+\zeta^2-1\]

So we can write:
\[\gamma_1 = \begin{bmatrix}
\zeta^9-\zeta^6-\zeta^3+1 & -\zeta^9-\zeta^6+\zeta^3-1\\
1 & -\zeta^6-1
\end{bmatrix}\]

Writing $\xi=-\zeta^3$ (a primitve $5$-th root), then $A$ is defined over $\mathbb{Q}(\xi)$:
\[\gamma_1 = \begin{bmatrix}
-\xi^3-\xi^2+\xi+1 & \xi^3-\xi^2-\xi-1\\
1 & -\xi^2-1
\end{bmatrix}\]

Correspondingly we have the polynomial:
\[f(x,y)=(-\xi^3-\xi^2+\xi+1)x+\xi^3-\xi^2-\xi-1-xy+(\xi^2+1)y\]
Interpreting the calculations in \cite{fu2022projective}, one can check that the curve $X=\{f=0\}$ contains the following pairs of roots of unity:
\begin{align*}
    X\cap \mu_{\infty}^{2} = &\{(1,1),(\zeta,\zeta^2),(\zeta^2,\zeta^5),(\zeta^3,\zeta^9),(\zeta^4,\zeta^{13}),(\zeta^5,\zeta^{16}),(\zeta^6,\zeta^{18}),(\zeta^9,\zeta^{21}),\\
    &(\zeta^{11},\zeta^{22}),(\zeta^{14},\zeta^{23}),(\zeta^{18},\zeta^{24}),(\zeta^{22},\zeta^{25}),(\zeta^{25},\zeta^{26}),(\zeta^{27},\zeta^{27})\}\\
    =& \{(z,\gamma_1(z)) \mid z \in S_1\}
\end{align*}

Again, we consider a selection of polynomials:
\begin{align*}
    &f_{1}(x,y)=f(x,-y) \\
    &f_{2}(x,y)=f(-x,y)\\
    &f_{3}(x,y)=f(-x,-y)\\
    &f_{4}(x,y)=f^{\sigma}(x^{2},y^{2})\\
    &f_{5}(x,y)=f^{\sigma}(x^{2},-y^{2}) \\
    &f_{6}(x,y)=f^{\sigma}(-x^{2},y^{2}) \\
    &f_{7}(x,y)=f^{\sigma}(-x^{2},-y^{2})
\end{align*}

Where ${\sigma}$ is given by:
\begin{align*}
    \sigma: \mathbb{Q}(\xi)&\longrightarrow \mathbb{Q}(\xi)\\
    \xi &\mapsto \xi^{2}
\end{align*}
Then one computes where the roots of unity on $X$ are distributed among the curves $X_{i}=\{f_{i}=0\}\subseteq (\mathbb{Q}(\xi)^*)^{2}$:
\begin{align*}
    &X\cap X_{1}=\emptyset \\
    &X\cap X_{2}=\emptyset \\
    &X\cap X_{3}=\{(\zeta^3,\zeta^9),(\zeta^{18},\zeta^{24})\} \\
    &X\cap X_{4}=\{(1,1),(\zeta^3,\zeta^9),(\zeta^6,\zeta^{18}),(\zeta^{18},\zeta^{24})\} \\
    &X\cap X_{5}=\{(\zeta^2,\zeta^5),(\zeta^4,\zeta^{13}),(\zeta^{14},\zeta^{23}),(\zeta^{22},\zeta^{25})\} \\
    &X\cap X_{6}=\{(\zeta,\zeta^2),(\zeta^5,\zeta^{16}),(\zeta^{11},\zeta^{22}),(\zeta^{25},\zeta^{26})\} \\
    &X\cap X_{7}=\{(\zeta^3,\zeta^9),(\zeta^9,\zeta^{21}),(\zeta^{18},\zeta^{24}),(\zeta^{27},\zeta^{27})\}
\end{align*}

Note that the 18 points in the above intersections overcount the number of roots of unity precisely by the pairs $(\zeta^3,\zeta^9),(\zeta^{18},\zeta^{24})$ appearing on $X_{3}$, $X_{4}$ and $X_{7}$.

\end{example}

\begin{example}
    
Let $\omega$ be a primitive 60-th root of unity. We consider the matrix:
\[\gamma_2 = \begin{bmatrix}
a & b\\
c & d
\end{bmatrix} = \begin{bmatrix}
-\omega^{14} - \omega^{12} - \omega^{10} + \omega^4 + \omega^2 + 1 & \omega^{14} + \omega^{12} + \omega^{10} - \omega^6 - \omega^4 - \omega^2\\
\omega^{12} + \omega^{10} + \omega^8 - \omega^2 & -\omega^{12} - \omega^{10} - \omega^8 - \omega^6 + \omega^2 + 1
\end{bmatrix}\]

Again, this matrix is defined over a subfield $\mathbb{Q}(\eta)$ where $\eta^{2}=-\omega$ is a primitive $15$-th root of unity:
\[\gamma_2 = \begin{bmatrix}
\eta^{7} + \eta^{6} + \eta^{5} - \eta^2 - \eta + 1 & -\eta^{7} - \eta^{6} - \eta^{5} + \eta^3 + \eta^2 + \eta\\
-\eta^{6} - \eta^{5} - \eta^4 + \eta & \eta^{6} + \eta^{5} + \eta^4 + \eta^3 - \eta + 1
\end{bmatrix}\]

Correspondingly we have a polynomial:
\begin{align*}
    g(x,y)=&(\eta^{7} + \eta^{6} + \eta^{5} - \eta^2 - \eta + 1)x-\eta^{7} - \eta^{6} - \eta^{5} + \eta^3 + \eta^2+ \eta \\
    &+(\eta^{6} + \eta^{5} + \eta^4 - \eta)xy-(\eta^{6} + \eta^{5} + \eta^4 + \eta^3 - \eta + 1)y
\end{align*}

Again, using the results of \cite{fu2022projective}, one can check that the curve $Y=\{g=0\}$ contains the following pairs of roots of unity:
\begin{align*}
    Y\cap \mu_{\infty}^{2} = &\{(1,1),(\omega,\omega^9),(\omega^{2},\omega^{18}),(\omega^{4},\omega^{28}),(\omega^6,\omega^{32}),(\omega^8,\omega^{34}),(\omega^{12},\omega^{36}), \\
    &(\omega^{22},\omega^{38}),(\omega^{31},\omega^{39}),(\omega^{40},\omega^{40}, (\omega^{50},\omega^{42}),(\omega^{54},\omega^{44}),(\omega^{56},\omega^{46}), (\omega^{58},\omega^{50})\}\\
    =& \{(z,\gamma_2(z)) \mid z \in S_2\}
\end{align*}

We now have a selection of polynomials:
\begin{align*}
    &g_{1}(x,y)=g(x,-y) \\
    &g_{2}(x,y)=g(-x,y)\\
    &g_{3}(x,y)=g(-x,-y)\\
    &g_{4}(x,y)=g^{\tau}(x^{2},y^{2})\\
    &g_{5}(x,y)=g^{\tau}(x^{2},-y^{2}) \\
    &g_{6}(x,y)=g^{\tau}(-x^{2},y^{2}) \\
    &g_{7}(x,y)=g^{\tau}(-x^{2},-y^{2})
\end{align*}

Where ${\tau}$ is given by:
\begin{align*}
    \tau: \mathbb{Q}(\eta)&\longrightarrow \mathbb{Q}(\eta)\\
    \eta &\mapsto \eta^{2}
\end{align*}
Again, one computes where the roots of unity on $Y$ are distributed among the curves $Y_{i}=\{g_{i}=0\}\subseteq (\mathbb{Q}(\omega)^{*})^{2}$:
\begin{align*}
    &Y\cap Y_{1}=\emptyset \\
    &Y\cap Y_{2}=\emptyset \\
    &Y\cap Y_{3}=\{(\omega,\omega^9),(\omega^{31},\omega^{39})\} \\
    &Y\cap Y_{4}=\{(1,1),(\omega^{40},\omega^{40}),(\omega^{12},\omega^{36}),(\omega^{4},\omega^{28})\} \\
    &Y\cap Y_{5}=\{(\omega^8,\omega^{34}),(\omega^{56},\omega^{46})\} \\
    &Y\cap Y_{6}=\{(\omega^6,\omega^{32}),(\omega^{54},\omega^{44})\} \\
    &Y\cap Y_{7}=\{(\omega^{22},\omega^{38}),(\omega^{50},\omega^{42}),(\omega^{2},\omega^{18}),(\omega^{58},\omega^{50})\}
\end{align*}

In contrast to Example 1, the above intersections consist of exactly 14 points in $(\mathbb{Q}(\omega)^{*})^{2}$. However, the curves $Y_{5}$ and $Y_{6}$ both intersect $Y$ in $(\mathbb{C}^{*})^{2}$ in two additional points which are not pairs of roots of unity.

\end{example}

\bibliographystyle{plain}
\bibliography{roots_of_unity.bib}

\end{document}